\documentclass{amsart}
\usepackage{amsmath,amsthm,latexsym,amssymb,hyperref, amsfonts}
\usepackage[T1]{fontenc}
\usepackage{stmaryrd}
\usepackage{eucal}
\usepackage{mathrsfs}
\usepackage[all]{xy}
\usepackage{color}
\usepackage{tikz}
\usepackage{url}
\usepackage{comment}
\usepackage{tikz-cd}
\usepackage{enumerate}
\usepackage{enumitem}
\usepackage{dutchcal}
\usepackage[normalem]{ulem}
\usetikzlibrary{matrix,arrows,decorations.pathmorphing}
\usepackage{quiver}

\numberwithin{equation}{section}

\theoremstyle{definition}
\newtheorem{defi}{Definition}[section]

\newtheorem{ex}[defi]{Example}
\newtheorem{rem}[defi]{Remark}
\newtheorem{notation}[defi]{Notation}
\newtheorem*{ack}{Acknowledgements}

\theoremstyle{plain}
\newtheorem{thm}[defi]{Theorem}
\newtheorem{prop}[defi]{Proposition}
\newtheorem{lem}[defi]{Lemma}
\newtheorem{cor}[defi]{Corollary}

\newcommand{\OO}{\mathcal{O}}

\newcommand{\colim}{colim}
\renewcommand{\lim}{lim}

\newcommand{\Ek}{\mathbb{E}_k}

\newcommand{\C}{\mathsf{C}}
\newcommand{\B}{\mathsf{B}}

\newcommand{\op}{\mathsf{op}}
\newcommand{\X}{\mathscr{X}} %%used to be \mathfrak{T}
\newcommand{\Xh}{\mathscr{Y}} %%the inf topos that truncates to R
\newcommand{\XX}{\mathscr{X}} %%used to be \mathfrak{R}
\newcommand{\YY}{\mathscr{Y}}
\newcommand{\Tinf}{\mathcal{S}}
\newcommand{\Cinf}{\mathcal{C}}
\newcommand{\Dinf}{\mathcal{D}}

\newcommand{\Catinf}{\mathcal{Cat}_\infty}
\newcommand{\Alginf}{\mathcal{Alg}}
\newcommand{\Grp}{\mathcal{Grp}}
\newcommand{\coalginf}{\mathcal{CoAlg}}

\newcommand{\comodinf}{\mathcal{CoMod}}
\newcommand{\lmodinf}{\mathcal{LMod}}
\newcommand{\rmodinf}{\mathcal{RMod}}

\newcommand{\Fun}{\mathcal{Fun}}

\long\def\emptytext#1{}

\title[Koszul duality in higher topoi]{Koszul duality in higher topoi}

\author[Jonathan Beardsley]{Jonathan Beardsley}
\address{Department of Mathematics and Statistics,
	University of Nevada, Reno,
	1664 N. Virginia Street, Reno, NV 89557, USA}
\email{jbeardsley@unr.edu}

 \author{Maximilien P\'eroux}
\address{Department of Mathematics,
University of Pennsylvania,
David Rittenhouse Lab.,
209 South 33rd Street, Philadelphia, PA 19104, USA}
\email{mperoux@sas.upenn.edu}

\subjclass [2010] {55U30, 18D35, 16T15}

\begin{document}

\begin{abstract}
	We show that there is an equivalence in any $n$-topos $\XX$ between the pointed and $k$-connective objects of $\XX$ and the $\Ek$-group objects of the $(n-k-1)$-truncation of $\XX$. This recovers, up to equivalence of $\infty$-categories, some classical results regarding algebraic models for $k$-connective, $(n-1)$-coconnective homotopy types. Further, it extends those results to the case of sheaves of such homotopy types.  We also show that for any pointed and $k$-connective object $X$ of $\X$ there is an equivalence between the $\infty$-category of modules in $\X$ over the associative algebra $\Omega^k X$, and the $\infty$-category of comodules in $\X$ for the cocommutative coalgebra $\Omega^{k-1}X$. All of these equivalences are given by truncations of Lurie's $\infty$-categorical bar and cobar constructions, hence the terminology ``Koszul duality.''
%	
%	In doing so, we prove several results which may be of broader interest: we show that Lurie's straightening-unstraightening equivalence holds over an $(n-1)$-groupoid in any $n$-topos for $0\leq n\leq\infty$; we show that taking categories of $\mathcal{O}$-algebras, for $\mathcal{O}$ an $\infty$-operad, commutes with truncation; and we show that for a pointed and connected $(n-1)$-groupoid $Y$ there is an equivalence between $\Omega Y$-modules in $\XX$ and $\XX$-valued presheaves on $Y$.
\end{abstract}

\maketitle

\section{Introduction}
Classical Koszul duality, sometimes called bar-cobar duality, gives an equivalence between suitably ``derived'' versions of algebras and coalgebras in a symmetric monoidal category equipped with some notion of ``homotopy theory'' (e.g.~a Quillen model category or an $\infty$-category). This general idea has many manifestations in algebra, topology and category theory. In the topological case, the duality relationship between algebras and coalgebras manifests as a relationship between pointed connected spaces $X$, which are cocommutative coalgebras via the diagonal map, and their associated loop spaces $\Omega X$, which are $\mathbb{A}_\infty$-algebras in the category of spaces. More specifically, it was shown in \cite[Section 13]{maygeom} that every grouplike $\mathbb{A}_\infty$-algebra in spaces is equivalent to the space of loops on a pointed and connected space. 

By working with simplicial sets, this is extended to a Quillen equivalence of model categories between simplicial groups and \textit{reduced} simplicial sets (which model pointed, connected spaces) in \cite[Theorem 6.3, Corollary 6.4]{goerssjardine}. In this work, we generalize this type of result in two different directions: first we show that such a result holds in an arbitrary $n$-topos (extending a result of Lurie's in $\infty$-topoi) as defined in Sections 6.4 and 6.1, respectively, of \cite{htt}; then we show that this ``categorifies'' to an equivalence of $\infty$-categories between modules over an algebra and comodules over its Koszul dual coalgebra.

This relationship between modules and comodules is already established in the literature for the category of spaces. Recall that there is an equivalence between spaces \textit{over} a space $X$, and comodules for the diagonal coalgebra structure of $X$ (this fact is relatively well known but also follows from our Corollary \ref{corollary: slicecomodules}). Then for an arbitrary space $X$, the duality between left $\Omega X$-modules and $X$-comodules in spaces is given in \cite[Theorem 2.1]{drordwyerkan} and \cite[Theorem 8.5]{shulmanparametrized} (for simplicial sets and topological spaces, respectively). There is also a pointed version of this equivalence given (after localizing) in \cite[Theorem 4.14]{hscomod}. 

In this paper, following \cite{htt,ha}, we will use the $\infty$-category of $\infty$-groupo\-ids, denoted $\Tinf$, to model topological spaces or simplicial sets. This $\infty$-category has the added benefit of being an $\infty$-topos. In fact, it is the canonical example of an $\infty$-topos and all other $\infty$-topoi behave similarly to $\Tinf$ in many important ways. The most common examples of $\infty$-topoi are $\infty$-categories of sheaves of spaces and certain localizations thereof. More generally, one can work with $n$-topoi for $0\leq n \leq\infty$ (cf.~\cite[\S 6.4]{htt}). The canonical $n$-topos is the $\infty$-category of $\infty$-groupoids whose homotopy groups are concentrated in degree $n-1$ and below, which we will denote by $\tau_{\leq n-1}\Tinf$. It is shown in \cite{htt} that every $n$-topos, for $n<\infty$, is in fact a category of sheaves of homotopy $n$-types on a site, which is not the case for $\infty$-topoi. Throughout this paper we will typically allow $n=\infty$, as many of our results hold for $n$-topoi with $n<\infty$ and $\infty$-topoi. 

An especially nice property of $n$-topoi for is that they admit all (small) limits and colimits, so they always have a symmetric monoidal structure via the categorical product. As a result, every object of an $n$-topos is a cocommutative coalgebra (cf. Corollary \ref{corollary: slicecomodules}). In Theorem \ref{thm: Koszul of gp like alg in Lurie}, which is a restatement of a result of Lurie, we recall the equivalence between pointed $k$-connective objects of an $\infty$-topos and $\Ek$-group objects of the same $\infty$-topos. To situate this as a Koszul duality result, we rephrase it as a relationship between algebras and coalgebras. This is Theorem \ref{cor: cobar in n-topoi} from which we obtain as examples some classical algebraic descriptions of connected and coconnected homotopy types \cite{brownGilbertAlgebraic, bullejosSymmetricCatGroups, garzonMirandaCatGroups} (though our results, being $\infty$-categorical, are manifestly less strict than the cited ones). See Examples \ref{ex: browngilbert} and \ref{ex: bullejos} in particular.

In our main result, Theorem \ref{theorem: mainthminfty}, we extend the equivalence between algebras and coalgebras to one between categories of modules and comodules \textit{over} those algebras and coalgebras in any $n$-topos. In particular since any $n$-topos $\X$ admits a finite limit preserving functor from the $n$-topos of $(n-1)$-groupoids $\pi^\ast\colon\tau_{\leq n-1}\Tinf\to \X$, every looping of an $(n-1)$-groupoid $\Omega X$ defines a group object  $\pi^\ast(\Omega X)\in\X$ and every $(n-1)$-groupoid $X$ defines a cocommutative coalgebra $\pi^\ast X\in\X$. So as a special case, our result gives an alternative ``coalgebraic'' description of the objects in any $n$-topos equipped with an action of a loop space. This special case can also be thought of as an $n$-toposic version of the the Grothendieck construction and its inverse (cf. Lemmas \ref{lemma: gengrothconstruction} and \ref{lemma: ngengrothconstruction}). This generalizes a theorem of Schreiber \cite[Theorem 3.4.20]{schreiberdifferentialcoh}.

\begin{ack} In many ways this paper represents an effort by the authors to become better acquainted with the ideas and methods of the theory of $\infty$-topoi.  We are very grateful to our colleagues and friends, who are more familiar with this material, for many interesting and helpful discussions. In particular, we would like to thank Ben Antieau, Omar Antol\'in-Camarena,  Alexander Campbell, Peter Haine, Rune Haugseng, Piotr Pstr\k{a}gowski, and Charles Rezk.
The first author was partially supported by NSF grant DMS-1745583 and a Simons Foundation Collaboration Grant, Award \#853272, while working on this paper.
%%As a result, many of the proofs below were sketched for us by colleagues and friends who are more familiar with this material. We thank:~Charles Rezk for outlining the main ideas of the proof of Lemma \ref{lemma: gengrothconstruction}; Peter Haine for pointing out \cite[Remark 5.2.6.28]{ha} and for helping us understand the material in \cite{htt} on $n$-topoi for $n<\infty$; Omar Antol\'in-Camarena for suggesting that $\infty$-topoi might be the easiest sorts of categories in which to prove the Koszul duality statements of interest to us; Rune Haugseng for pointing out that we could use \cite[Corollary 2.6.6]{haugsengshifted} to prove Lemma \ref{lemma: cocartbimodules} in the $\infty$-categorical setting; Piotr Pstr\k{a}gowski for pointing out that truncation commutes with taking algebras over a monad as described in Lemma \ref{lem:monadalgebratruncation}; and Alexander Campbell for many helpful conversations related to this work. We would also like to thank Ben Antieau for pointing out an embarrassing off-by-one error present in an early draft of this paper.   
\end{ack}

\begin{notation} We begin by setting some notation that we will use throughout and recalling some elementary notions of the theory of $\infty$-categories and $\infty$-topoi.
	\begin{enumerate} 
	\item We use the much of the notation and terminology of \cite{htt} and \cite{ha}. In particular, we write $\Catinf$ and $\Tinf$ to refer to the $\infty$-categories of small $\infty$-categories and $\infty$-groupoids, respectively. 
	
	\item If $\Cinf$ is an $\mathcal{O}$-monoidal $\infty$-category then we will write $\mathcal{Alg}_{\mathcal{O}}(\Cinf)$ for the $\infty$-category of $\mathcal{O}$-algebras in $\Cinf$. In the case that $\mathcal{O}=\mathbb{E}_1\simeq \mathbb{A}_\infty$, we will just write $\mathcal{Alg}(\Cinf)$. 
	
	\item If $\Cinf$ is a presentable $\infty$-category then it admits all (small) colimits, and so is tensored over $\infty$-groupoids, as described in \cite[Remark 5.5.1.7]{htt}. As such, given any $\mathbb{A}_\infty$-algebra $A$ in $\Tinf$, we can consider left modules over $A$ in $\Cinf$ as in \cite[Definition 4.2.1.13]{ha}. We will denote this $\infty$-category by $\lmodinf_{A}(\Cinf)$. 
	
	\item Let $\Cinf$ be a $\mathcal{O}$-monoidal $\infty$-category, as in \cite[2.1.2.15]{ha}, where $\mathcal{O}^\otimes$ is an $\infty$-operad. Recall that there is an induced $\mathcal{O}$-monoidal structure on the opposite $\infty$-category $\Cinf^\op$, see \cite[2.4.2.7]{ha}. An $\mathcal{O}$-coalgebra in $\Cinf$ is an $\mathcal{O}$-algebra in $\Cinf^\op$. We denote the $\infty$-category of $\mathcal{O}$-coalgebras by $\coalginf_\mathcal{O}(\Cinf):=\mathcal{Alg}_\mathcal{O}(\Cinf^\op)^\op$. 
	Given an $\mathcal{O}$-coalgebra $C$ in $\Cinf$, we denote the $\infty$-category of left comodules over $C$ in $\Cinf$ by:
	\[
	\mathcal{LCoMod}_C(\Cinf):=\mathcal{LMod}_C(\Cinf^\op)^\op,
	\]
	and similarly, we use $\mathcal{RCoMod}_C(\Cinf)$ to denote right comodules.
	
	\item We will be especially interested in the little $k$-cube $\infty$-operads $\Ek$. We will always assume that $k>0$ when considering these $\infty$-operads. 
	
	\item When we say (co)associative (co)algebras and (co)commutative (co)al\-gebras (or monoids) we will mean $\mathbb{A}_\infty$-(co)algebras and $\mathbb{E}_\infty$-(co)algebras respectively.
	
	\item We will make use of colimits and limits over constant diagrams, so to simplify notation, whenever we wish to denote the colimit or limit of a constant diagram at $U$ over an indexing $\infty$-category $\Cinf$, we will write $\colim_{\Cinf}U$ or $\lim_{\Cinf} U$, respectively. The $\infty$-category within which this colimit or limit is being taken will always be clear from context.

	\item If an $\infty$-category $\Cinf$ has a terminal object, we will denote this terminal object by $1_\Cinf$. In particular, we will write $1_{\Tinf}$ for a contractible $\infty$-groupoid.
	
	\item We use the terms $n$-topos and $\infty$-topos in the sense of \cite[\S 6]{htt}. These are $\infty$-categories that behave like the $\infty$-category of $(n-1)$-groupoids and $\infty$-groupoids respectively. These $\infty$-categories are a generalization of what are typically called Grothendieck topoi in 1-category theory (as opposed to elementary topoi). It would be more precise to use the terminology $(\infty,1)$-topoi and $(n,1)$-topoi, but we follow \cite{htt} in omitting the second index, as every category in this paper will only have invertible $k$-morphisms for $k>1$.

	\item For a pointed object $X$ in an $\infty$-topos $\X$, we will write $\Omega X$ to denote the ``loop space'' object of $X$, i.e.~the pullback of the cospan $1_{\X}\to X\leftarrow 1_{\X}$ determined by the pointing.

	\item Any $\infty$-topos $\X$ admits a unique \textit{geometric morphism} $\pi\colon \X\to \Tinf$ (cf. \cite[\S 6.3.1]{htt} and \cite[Proposition 6.3.4.1]{htt}). This implies that there is a functor $\pi^\ast\colon \Tinf\to \X$ which preserves colimits and finite limits. Therefore, for any Kan complex $X$ there is an object $\pi^\ast(X)\in\X$ which we think of as $X$ pulled back along $\pi$. Because the functor $\pi^\ast$ preserves colimits and terminal objects and any $\infty$-groupoid may be written as $X\simeq \colim_X1_{\Tinf}$, we have that $\pi^\ast(X)\simeq colim_X1_{\X}$.
	
	\item If $\XX$ is an $n$-topos then it can be written as a truncation of an $\infty$-topos (cf.~\cite[Theorem 6.4.1.5 (2)]{htt}). Thus by the number (9) above and \cite[Lemma 6.4.5.5]{htt} (taking $\mathcal{C}=\Tinf$) there is a unique geometric morphism of $n$-topoi $\pi\colon \XX\to \tau_{\leq n-1}\Tinf$, where $\tau_{\leq n-1}\Tinf$ is the $\infty$-category of $(n-1)$-groupoids (i.e.~$\infty$-groupoids with trivial homotopy groups above degree $n-1$). Thus every $(n-1)$-groupoid $X$ can be pulled back to an object $\pi^\ast(X)\simeq colim_{X}1_{\XX}\in\XX$.
		\end{enumerate}
\end{notation}

\section{Koszul Duality in Higher Topoi for Coalgebras and Algebras}

Let $\Cinf$ be a monoidal $\infty$-category, which admits both totalizations of cosimplicial objects and geometric realizations of simplicial objects.   Then, recall from  \cite[5.2.3.6, 5.2.3.9]{ha}, there is a \emph{bar-cobar adjunction} between augmented $\Ek$-algebras and coaugmented $\Ek$-coalgebras in $\Cinf$ for $k\geq 1$:
\[
\begin{tikzcd}
\mathcal{Alg}_{\Ek}^{aug}(\Cinf) \ar[shift left=2]{r}{\B^k}[swap]{\perp} & \coalginf_{\Ek}^{coaug}(\Cinf)\ar[shift left=2]{l}{\C^k},
\end{tikzcd}
\]
induced by the iterated bar and cobar constructions. This adjunction is in general not an equivalence of $\infty$-categories.

We shall be interested in the cases where $\Cinf$ is endowed with its Cartesian symmetric monoidal structure. In that situation, \cite[Proposition 2.4.2.5]{ha} gives an equivalence $\mathcal{Alg}_{\Ek}^{aug}(\Cinf)\simeq \mathcal{Mon}_{\Ek}(\Cinf)$ where the right hand side refers to the $\infty$-category of $\Ek$-monoid objects described in \cite[2.4.2]{ha} for the $\infty$-operad $\Ek^\otimes$. Note that it is superfluous to say ``augmented monoid'' since all algebra objects, equivalently all monoid objects, are augmented in a Cartesian monoidal $\infty$-category.
%%In this case, as noted in \cite[5.2.6.12]{ha}, the $k$-fold cobar construction $\C^k$ is equivalent to the $k$-fold loop space $X \mapsto \Omega^k X$.  

On the other hand, in any $\infty$-category equipped with the Cartesian monoidal structure, every object admits a unique cocommutative coalgebra structure. This is described explicitly, along with its implications for $\infty$-categories of modules and comodules, by the following lemma and its corollary.

\begin{lem}\label{lemma: cocartbimodules}
	If $\Cinf$ is a coCartesian symmetric monoidal $\infty$-category, then every object of $\Cinf$ is an $\Ek$-algebra in a homotopically unique way for every $0\leq k\leq\infty$. Additionally, for any $X\in\Cinf$ there is an equivalence $\lmodinf_{X}(\Cinf)\simeq\rmodinf_{X}(\Cinf)\simeq \Cinf^{\backslash X}$, where $\Cinf^{\backslash X}$ denotes the $\infty$-category of objects under $X$ in $\Cinf$. 
\end{lem}

\begin{proof}
	It is shown in \cite[Corollary 2.4.3.9]{ha} that for a unital $\infty$-operad $\OO^\otimes$ there is an equivalence $\Alginf_{\OO}(\Cinf)\simeq Fun(\OO,\Cinf)$ where $\OO$, as an $\infty$-category, is the fiber over the set $\{\ast, 1\}$ of the fibration $\OO^{\otimes}\to N(\mathcal{Fin}_\ast)$ defining the $\infty$-operad structure on $\OO^\otimes$. For $\Ek^\otimes$ the underlying $\infty$-category is precisely $\Delta^0$, the terminal $\infty$-category. Therefore $\Alginf_{\Ek}(\Cinf)\simeq \Cinf$, i.e.~ every object of $\Cinf$ is a uniquely an $\Ek$-algebra. The equivalence $\lmodinf_{X}(\Cinf)\simeq\rmodinf_{X}(\Cinf)$ follows from \cite[Corollaries 4.5.1.6,  5.1.4.11]{ha} (where we are using the $\mathbb{E}_1$-structure on $X$ to define the $\infty$-category of left and right modules).
	
From \cite[Corollary 2.6.6]{haugsengshifted}, we have that for any cocartesian symmetric monoidal $\infty$-category $\Cinf$, there is an equivalence between algebras in $\Cinf$ for the bimodule $\infty$-operad $\mathcal{BM}^\otimes$ (whose algebras are triples $(A,M,B)$ where $A$ and $B$ are $\mathbb{A}_\infty$-algebras and $M$ is an $(A,B)$-bimodule, as described in \cite[Section 4.3.2]{ha}) and functors $\Fun(\Sigma^{1,\op},\Cinf)$ where $\Sigma^{1}$ is the (nerve of the) span category $\bullet\leftarrow\bullet\rightarrow\bullet$. 
The algebras of $\mathcal{BM}^\otimes$ for fixed algebras $A$ and $B$ are denoted ${}_A\mathcal{BMod}_B(\Cinf)$. 
Thus we have an equivalence between ${}_A\mathcal{BMod}_B(\Cinf)$ and diagrams in $\Cinf$ of the form $A\to M\leftarrow B$ (note that this does not follow from \cite[Corollary 2.4.3.9]{ha} as above because it is ultimately a statement about non-symmetric $\infty$-operads).  
Using \cite[Corollary 2.4.3.10]{ha} and recalling that the tensor unit of $\Cinf$ is the initial object, we obtain for any object $X\in \Cinf$ an equivalence: \[\lmodinf_X(\Cinf)\simeq{}_X\mathcal{BMod}_{1_\Cinf}(\Cinf)\simeq  \Cinf^{\backslash X}.\qedhere\]
\end{proof}

\begin{cor}\label{corollary: slicecomodules}
	Let $\Cinf$ be a Cartesian symmetric monoidal $\infty$-category. Then every object of $\Cinf$ is an $\Ek$-coalgebra in a unique way for all $0\leq k\leq\infty$ and for any $X\in\Cinf$ there are equivalences $\mathcal{LCoMod}_{X}(\Cinf)\simeq \mathcal{RCoMod}_{X}(\Cinf)\simeq \Cinf_{/X}$ where $\Cinf_{/X}$ denotes the $\infty$-category of objects over $X$ in $\Cinf$. 
\end{cor}

\begin{proof}
Apply Lemma \ref{lemma: cocartbimodules} to $\Cinf^{\op}$. 
\end{proof}

\begin{rem}\label{rem: jequalsk}
Notice that the proof of Lemma \ref{lemma: cocartbimodules} implies that for any $0\leq j\leq k\leq\infty$ there is a sequence of equivalences $\Cinf\simeq\Alginf_{\Ek}(\Cinf)\simeq\Alginf_{\mathbb{E}_j}(\Cinf)$ and the same holds for coalgebras. Therefore we have an identification $\coalginf_{\mathbb{E}_\infty}(\Cinf)\simeq\coalginf_{\Ek}(\Cinf)\simeq \Cinf$ when $\Cinf$ is Cartesian monoidal.  Moreover, by using these identifications, its easy to see that \textit{coaugmented} $\Ek$-coalgebras, for any $0\leq k\leq\infty$, are precisely \textit{pointed} objects of $\Cinf$. Going forward, whenever $\Cinf$ is Cartesian monoidal, we will identify $\Cinf_\ast$ and $\coalginf_{\mathbb{E}_\infty}^{coaug}(\Cinf)$ without comment.
\end{rem}

\noindent In light of the above, we can rewrite the adjunction from the beginning of this section as follows:

\[
\begin{tikzcd}
\mathcal{Mon}_{\Ek}(\Cinf) \ar[shift left=2]{r}{\B^k}[swap]{\perp} & \Cinf_\ast\ar[shift left=2]{l}{\C^k}.
\end{tikzcd}
\]

\begin{rem}
	%%In the above Lemma, $\Omega X$ denotes the cobar construction on $X$ as in \cite[Notation 5.2.6.11]{ha}.
	As noted in \cite[5.2.6.12]{ha}, if we compose the iterated cobar construction $\C^k$ on a Cartesian monoidal $\infty$-category with the forgetful functor $\mathcal{Mon}_{\Ek}(\Cinf)\to \Cinf$, we obtain the ``iterated loop space'' construction $X \mapsto \Omega^k X$. Note that, in the case that $\Cinf\simeq \Tinf$, the object $\Omega^k X$ is equivalent to the usual iterated loop space of $X$ \cite[Remark 5.2.6.12]{ha}. Moreover, if $\Cinf\simeq \X$ for an $\infty$-topos $\X$ then the pullback functor $\pi^\ast\colon \Tinf\to \X$ preserves finite limits, so we have that $\Omega^k \pi^\ast(Y)\simeq \pi^\ast(\Omega^k Y)$ for any $\infty$-groupoid $Y$. 
\end{rem}

In \cite[5.2.6]{ha}, Lurie demonstrates that the bar-cobar adjunction fails to be an equivalence, even in the $\infty$-topos of spaces, for two reasons:
\begin{enumerate}
	\item A map $f:X\rightarrow Y$ of pointed spaces induces a weak homotopy equivalence $\Omega^k X \rightarrow \Omega^k Y$ as long as $f$ induces an isomorphism on homotopy groups in degree $k$ and higher. We thus need to consider \emph{$k$-connective} objects on the side of coalgebras if we want the bar-cobar adjunction to be an equivalence. The notion of being $k$-connective extends naturally to any $\infty$-topos (cf. \cite[Definition 6.5.1.11]{htt}) and we give the generalization of this to $n$-topoi in Definition \ref{def: connective for n-topoi}.
	
	\item If $Y$ is an $\Ek$-monoid in spaces, then its multiplication induces a monoid structure on its path-components $\pi_0(Y)$. An equivalence $Y\simeq \Omega X$ would imply that $\pi_0(Y)\cong \pi_1(X)$, making $\pi_0(Y)$ into a group. We therefore need to consider \emph{grouplike $\Ek$-monoids} in the sense of \cite[5.2.6.6]{ha} on the side of algebras in the bar-cobar adjunction. These can be thought of as objects whose ``space of connected components'' forms a group with respect to the monoid structure.
\end{enumerate}

By accounting for the above issues, Lurie obtains an equivalence in $\infty$-topoi equipped with the Cartesian symmetric monoidal structure. Going forward we will use $\mathcal{Grp}_{\Ek}(\Cinf)$ to denote the $\infty$-category of grouplike $\Ek$-monoids in an $\infty$-category $\Cinf$ with finite products and refer to the objects of $\mathcal{Grp}_{\Ek}(\Cinf)$ as ``$\Ek$-group objects.'' In this language, Lurie's theorem can be restated as follows.

\begin{thm}[{\cite[5.2.6.15]{ha}}]\label{thm: Koszul of gp like alg in Lurie}
	Let $\X$ be an $\infty$-topos, endowed with the Cartesian symmetric monoidal structure. Then the cobar construction induces an equivalence between pointed $k$-connective objects in $\X$ and $\Ek$-group objects in $\X$: 
	\[
	\mathsf{C}^k\colon \X_*^{\geq k}\stackrel{\simeq}\longrightarrow \mathcal{Grp}_{\Ek}(\X).
	\]
	Moreover, the homotopy inverse, of $\C^k$ is the iterated bar construction $\B^k$. 
\end{thm}

%%\begin{rem}
%%It is not stated explicitly in \cite[5.2.6.15]{ha} that $\B^k$ is homotopy inverse to $\C^k$, but it follows from the fact that every equivalence of $\infty$-categories can be extended to an adjoint equivalence and that adjoints are unique up to homotopy. 
%%\end{rem}

Our goal is to prove an analogue of this theorem for $n$-topoi for $0\leq n<\infty$. Before stating and proving our result as Theorem \ref{cor: cobar in n-topoi}, it will be necessary to extend certain definitions and lemmas from \cite{htt} to the setting of $n$-topoi. Note that many of our results hold for $(-1)$-topoi as well, but up to equivalence there is only one $(-1)$-topos, the terminal $\infty$-category, and so we ignore this case.

\begin{defi}\label{definition: truncated}
	Let $\Cinf$ be an $\infty$-category and $-2\leq n\leq\infty$. Following \cite[Definition 5.5.6.1]{htt} we say that an object $X$ of $\Cinf$ is \emph{$m$-truncated} if for every object $Y\in \Cinf$ the mapping space $\Cinf(Y,X)$ is an $m$-truncated space. We denote the full subcategory of $m$-truncated objects of $\Cinf$ by $\tau_{\leq m}\Cinf$. 
\end{defi}

\begin{rem}
	Recall that if $m\geq 0$, then an $m$-truncated space is precisely a space which has trivial homotopy groups in degrees greater than $m$. Otherwise we say a space is $(-2)$-truncated if it is contractible and non-empty and $(-1)$-truncated if it is either contractible or empty. Notice that this implies that a $(-2)$-truncated object of an $\infty$-category must always be terminal. Some of our results will concern $m$-truncated objects where it is possible that $m<-2$. In that situation we will make the convention to take $\tau_{\leq m}\Cinf$ to be the full subcategory of $(-2)$-truncated objects in $\Cinf$.
\end{rem}

\begin{prop}
	Let $\X$ be an $n$-topos for $0\leq n< \infty$. Then every object of $\X$ is $(n-1)$-truncated.
\end{prop}

\begin{proof}
	The claim follows from \cite[Proposition 6.4.5.7]{htt} which states that any $n$-topos is equivalent to the full subcategory of $(n-1)$-truncated objects of an $\infty$-topos (and being $(n-1)$-truncated is stable under equivalence of $\infty$-categories). 
\end{proof}

\begin{rem}
In what follows we will very often use the fact that every $n$-topos $\XX$ is equivalent to $\tau_{\leq n-1}\YY$ for an $\infty$-topos $\YY$ (cf.~\cite[Theorem 6.4.1.5 (2)]{htt}). This allows us to reduce the proofs of our results to understanding the way in which truncation interacts with certain algebraic structures on $\infty$-topoi. 
\end{rem}

\begin{prop}\label{prop: truncfunc}
Let $\X$ be an $n$-topos for $0\leq n\leq\infty$ and let $\tau_{\leq m}\X$ be the full subcategory of $m$-truncated objects of $\X$ for $m\geq -2$. Then there is a finite product preserving left adjoint to the inclusion $\tau_{\leq m}\X\subseteq\X$. 
\end{prop}

\begin{proof}
	In the case of $n=\infty$, this is the content of \cite[Proposition 5.5.6.18, Lemma 6.5.1.2]{htt}. In the case of $n<\infty$, the fact that the inclusion $\tau_{\leq m}\X\subseteq\X$ admits a left adjoint still follows from \cite[Proposition 5.5.6.18]{htt} (which holds for all $\infty$-categories) but we need to be slightly more careful with asking that this adjoint preserve products, as \cite[Lemma 6.5.1.2]{htt} only applies to $\infty$-topoi as written. However, we can repeat the proof of the above cited lemma for $n$-topoi, making the appropriate changes, as they are also left exact localizations of presheaf $\infty$-categories valued in $(n-1)$-truncated spaces (by \cite[Definition 6.4.1.1]{htt}).
\end{proof}

\begin{rem}
	There are several edge cases in the above proposition that it may be helpful to clarify. In the case that $m\geq n-1$, both the inclusion $\tau_{\leq m}\X\subseteq\X$ and its left adjoint are the identity functor. In the case that $m=-2$, the $\infty$-category $\tau_{\leq m}\X$ is the terminal $\infty$-category (the unique $(-1)$-topos) and the inclusion $\tau_{\leq m}\X\subseteq\X$ is the functor picking out the terminal object of $\X$ whose left adjoint is the terminal functor. In the case that $m=-1$, $\tau_{\leq m}\X$ is the poset of subobjects of the terminal object of $\X$ (recall from \cite[Section 6.4.2]{htt} that $0$-topoi are precisely \textit{locales}). 
\end{rem}

\begin{defi}
	We write $\tau_{\leq m}\colon \X\to \tau_{\leq m}\X$ for the left adjoint to the inclusion $\tau_{\leq m}\X\subseteq \X$ obtained from Proposition \ref{prop: truncfunc}. When there is any chance of confusion we will write $\tau_{\leq m}^{\XX}$, as we will need to consider truncation functors of different $\infty$-categories at the same time. 
\end{defi}

\begin{rem}\label{rem:truncfunccommute}
We will often use the fact that ``truncation functors commute.'' In other words, given a presentable $\infty$-category $\Cinf$ and any $m,n\geq -2$ there is an equivalence $\tau_{\leq m}^{\tau_{\leq n}\Cinf}\tau_{\leq n}\Cinf\simeq\tau_{\leq n}^{\tau_{\leq m}\Cinf}\tau_{\leq m}\Cinf$. This follows from the definition of truncation and uniqueness of adjoints.
As a result, in what follows we will never write $\tau_{\leq m}^{\tau_{\leq n}\Cinf}$ to indicate the $m$-truncation functor applied to the $\infty$-category of $n$-truncated objects, but just write $\tau_{\leq m}$ (or $\tau_{\leq m}^{\Cinf}$) regardless of whether it is being applied to $\Cinf$ or $\tau_{\leq n}\Cinf$. 
\end{rem}

\begin{rem}
For any presentable $\infty$-category $\Cinf$, it is always true that the inclusion of the $m$-truncated objects of $\Cinf$ admits a left adjoint. However, Proposition \ref{prop: truncfunc} implies that whenever $\Cinf$ is an $n$-topos, that left adjoint is symmetric monoidal with respect to Cartesian symmetric monoidal structure, which will be essential going forward.
\end{rem}

\begin{defi}\label{def: connective for n-topoi}
	Let $\X$ be an $n$-topos for $0\leq n\leq\infty$. Then we say that an object $X\in\X$ is $m$-\emph{connective} if $\tau_{\leq m-1}X$ is a terminal object of $\X$. We denote the $m$-connective objects of $\X$ by $\X^{\geq m}$. We sometimes say \textit{connected} in place of $1$-connective. 
\end{defi}

\begin{rem}
	In the case that $\XX$ is an $\infty$-topos, our Definition \ref{def: connective for n-topoi} of $k$-connective objects is different than that given in \cite[Definition 6.5.1.10]{htt} but is equivalent as a result of \cite[Proposition 6.5.1.12]{htt}.
\end{rem}

\begin{lem}\label{lemma: truncation preserves connective}
	Let $\X$ be an $n$-topos for $0\leq n\leq\infty$ and let $X\in\X$ be $k$-connective. Then $\tau_{\leq m}X$ is $k$-connective for any $m\geq -2$. 
\end{lem}

\begin{proof}	
	
	Because $X$ is $k$-connective, we have that $\tau_{\leq k-1}X$ is terminal in $\X$. The result follows from noticing that $\tau_{\leq m}$ is a left adjoint, so it preserves terminal objects, and that truncations commute as noticed in Remark \ref{rem:truncfunccommute}.  
\end{proof}

\begin{lem}\label{lem: truncatedequivalence}
Suppose that $F\colon\Cinf\to\mathcal{D}$ and $G\colon\mathcal{D}\to\Cinf$ are inverse equivalences of presentable $\infty$-categories. Then for all $n$ there are induced functors $\tau_{\leq n}F\colon \tau_{\leq n}\Cinf\to\tau_{\leq n}\mathcal{D}$ and $\tau_{\leq n}G\colon \tau_{\leq n}\mathcal{D}\to\tau_{\leq n}\Cinf$ which are also inverse equivalences. 
\end{lem}

\begin{proof}
The existence of the restricted functors follows from \cite[5.5.6.16]{htt} and the fact that equivalences always preserve all small limits and colimits (by \cite[5.3.2.4]{htt} and its dual). The result then follows from \cite[5.5.6.28]{htt}.  
\end{proof}

\begin{lem}\label{lem: loopsBshifting}
Let $\XX$ be an $\infty$-topos and let X be a pointed and $n$-truncated object in $\XX$. Then $\Omega^k X$, the underlying object of $\C^kX$, is $(n-k)$-truncated. On the other hand, if $G\in \mathcal{Grp}_{\Ek}(\XX)$ and $G$ is $n$-truncated as an object of $\XX$ then $\B^kG$ is $(n+k)$-truncated as an object of $\XX$. 
\end{lem}

\begin{proof}
For the first claim, suppose that $X\in\XX_\ast$ and that $X$ is $n$-truncated as an object of $\XX$. Then by definition we have that $\XX(Y,X)$ is an $n$-truncated space for all $Y\in\XX$. Moreover, we have $\XX(Y,\Omega^k X)\simeq \Omega^k \XX(Y,X)$, which is $(n-k)$-truncated. Therefore $\Omega^k X$ is $(n-k)$-truncated. 

On the other hand, suppose $G\in\mathcal{Grp}_{\Ek}(\XX)$ and that $G$ is $n$-truncated as an object of $\XX$. From the equivalence in Theorem \ref{thm: Koszul of gp like alg in Lurie}, we have that $\Omega^k \B^kG\simeq G$. Therefore $\XX(Y, G)\simeq \XX(Y,\Omega^k \B^kG)\simeq \Omega^k\XX(Y,\B^kG)$ is $n$-truncated, from which it follows that $\XX(Y, \B^k G)$ must be $(n+k)$-truncated (since it is a space). 
\end{proof}

\begin{lem}\label{lem:monadalgebratruncation}
	Let $\Cinf$ be an $\infty$-category and $T\colon \Cinf\to\Cinf$ a monad on $\Cinf$ with $\infty$-category of algebras $\Alginf_T(\Cinf)$. Then there is an equivalence: 
	\[
	\tau_{\leq n}\Alginf_T(\Cinf)\simeq \Alginf_T(\Cinf)\times_{\Cinf}\tau_{\leq n}\Cinf,
	\]
	where the right hand side is the pullback in $\Catinf$ of the forgetful functor $\Alginf_T(\Cinf)\to\Cinf$ along the inclusion $\tau_{\leq n}\Cinf\hookrightarrow\Cinf$.
\end{lem}

\begin{rem}
For aesthetic reasons, we will denote the $\infty$-category $\Alginf_T(\Cinf)\times_{\Cinf}\tau_{\leq n}\Cinf$ by $\Alginf_T(\Cinf)_{\leq n}$. This is the full subcategory of $\Alginf_T(\Cinf)$ on the objects which are $n$-truncated as objects of $\Cinf$. Note that it does not automatically follow that $\Alginf_T(\Cinf)_{\leq n}\simeq \Alginf_T(\tau_{\leq n}\Cinf)$ (indeed we don't even know that $T$ always descends to a monad on $\tau_{\leq n}\Cinf$). However in the cases of interest to us this will be true, as a result of Lemma \ref{lem: lmodequiv}. 
\end{rem}

\begin{proof}
	
	Note that both sides of the equivalence are full subcategories of $\Alginf_T(\Cinf)$ (since fully faithful functors are stable under pullback by \cite[24.12]{joyalnotes}) so it suffices to show that they have the same objects. Suppose then that $X\in Alg_T(\Cinf)$ and that $X$ is $n$-truncated as an object of $\Alginf_T(\Cinf)$, i.e.~$X\in\tau_{\leq n}\Alginf(\Cinf)$. We will write $F_T$ for the free $T$-algebra functor which is left adjoint to the forgetful functor $U_T\colon\Alginf_T(\Cinf)\rightarrow \Cinf$.  
	Then for any $T$-algebra $Y$ the space $\Alginf_T(\Cinf)(Y,X)$ is $n$-truncated. In particular, $\Alginf_T(\Cinf)(F_T(Z),X)$ is $n$-truncated for all $Z\in\Cinf$. Therefore, since $F_T$ is left adjoint to $U_T$, the space $\Cinf(Z,U_T(X))$ is $n$-truncated for all $Z\in\Cinf$. In other words, the $T$-algebra $X$ is $n$-truncated as an object of $\Cinf$, so $X\in \Alginf_{T}(\Cinf)_{\leq n}$. 
	
	On the other hand, suppose $X\in \Alginf_T(\Cinf)$ and $U_T(X)$ is $n$-truncated, i.e.~$X\in \Alginf_{T}(\Cinf)_{\leq n}$. As a result of  \cite[Lemma 4.7.3.14]{ha} we have that every $T$ algebra is the colimit of a simplicial object which is levelwise a free $T$-algebra (note that an explicit description of this simplicial object, which is the classical monadic bar resolution, can be extracted from \cite[Example 4.7.2.7]{ha}). If $Y$ is a $T$-algebra in $\Cinf$ denote its canonical simplicial resolution by $\B_T(Y)_\bullet$ where $\B_T(Y)_i$ is a free $T$-algebra. Then we have: 
	\[ \Alginf_T(\Cinf)(Y,X)\simeq \Alginf_T(\Cinf)(\colim_i \B_T(Y)_i,X)\simeq \lim_i \Alginf_T(\Cinf)(\B_T(Y)_i,X).  \]
	By the same argument given in the first half of the proof, $\Alginf_T(\Cinf)(\B_T(Y)_i,X)$ is $n$-truncated for each $i$. Finally, $n$-truncated spaces are closed under all small limits (as a result of \cite[Proposition 5.5.6.5]{htt} and the fact that $\Tinf$ is complete), giving the result. 
\end{proof}

\begin{lem}\label{lem: lmodequiv}
Let $\Cinf$ be a monoidal $\infty$-category and $A\in\Alginf(\tau_{\leq n})\subseteq\Alginf(\Cinf)$. Then there is an equivalence of $\infty$-categories $\tau_{\leq n}\lmodinf_A(\Cinf)\simeq \lmodinf_A(\tau_{\leq n}\Cinf)$.
\end{lem}

\begin{proof}
	Consider the commutative square of $\infty$-categories
	
	\[
	\begin{tikzcd}
		\lmodinf(\tau_{\leq n}\Cinf)\arrow[r,hookrightarrow]\ar[d] & \lmodinf(\Cinf)\ar[d]\\
		\Alginf(\tau_{\leq n}\Cinf)\ar[r, hookrightarrow] & \Alginf(\Cinf)
	\end{tikzcd}
	\]
	in which the horizontal functors are inclusions of full subcategories. Recall that beacuse fully faithful functors of $\infty$-categories form the right class of an orthogonal factorization system on $\Catinf$ (cf.~\cite[24.12]{joyalnotes}) they are stable under all limits in the arrow $\infty$-category $\Catinf^{\Delta^1}$, by \cite[Proposition 5.2.8.6]{htt}. Therefore, pulling back the upper horizontal morphism in the preceding diagram along the morphism (in the arrow category)
	\[
	\begin{tikzcd}
		\{A\}\ar[r]\ar[d,equal] & \Alginf(\tau_{\leq n}\Cinf)\ar[d, hookrightarrow]\\
		\{A\}\ar[r] & \Alginf(\Cinf)
	\end{tikzcd}
	\]
	yields a fully faithful inclusion $\lmodinf_{A}(\tau_{\leq n}\Cinf)\hookrightarrow \lmodinf_{A}(\Cinf)$. Here we are using the fact that the $\infty$-category $\lmodinf_A(\Cinf)$ (resp.~$\lmodinf_A(\tau_{\leq n}\Cinf)$) is defined to be the pullback of the Cartesian fibration $\lmodinf(\Cinf)\to \Alginf(\Cinf)$ (resp.~$\lmodinf(\tau_{\leq n}\Cinf)\to \Alginf(\tau_{\leq n}\Cinf)$) along the inclusion of $A$ (cf.~\cite[Definition 4.2.1.13]{ha}).
	
	On the other hand, by Lemma \ref{lem:monadalgebratruncation}, we have the pullback square
	
	\[
	\begin{tikzcd}
	\tau_{\leq n}\lmodinf_{A}(\Cinf)\ar[r,hookrightarrow]\ar[d] & \lmodinf_{A}(\Cinf)\ar[d]\\
	\tau_{\leq n}\Cinf\ar[r,hookrightarrow] & \Cinf
	\end{tikzcd}
	\]
	in which the upper horizontal functor is fully faithful because fully faithful functors are stable under pullback (again by combining \cite[24.12]{joyalnotes} with \cite[Proposition 5.2.8.6]{htt}). Therefore we have two full subcategories of $\lmodinf_A(\Cinf)$, namely $\tau_{\leq n}\lmodinf_{A}(\Cinf)$ and $\lmodinf_{A}(\tau_{\leq n}\Cinf)$ which will be equivalent so long as they contain the same objects. From the equivalence $\tau_{\leq n}\lmodinf(\Cinf)\simeq \lmodinf_A(\Cinf)\times_{\Cinf}\tau_{\leq n}\Cinf$, we know that an object in $\tau_{\leq n}\lmodinf(\Cinf)$ is exactly the data of a morphism of $\infty$-operads $\mathcal{LM}^\otimes\to \Cinf^\otimes$ which takes $\mathfrak{a}\in\mathcal{LM}^\otimes$ to $A$ and $\mathfrak{m}\in \mathcal{LM}^\otimes$ to an object of $\tau_{\leq n}\Cinf$. Because $A\in\tau_{\leq n}\Cinf$, this is the same data of a morphism of $\infty$-operads $\mathcal{LM}^\otimes \to \tau_{\leq n}\Cinf^\otimes$. 
\end{proof}

\begin{rem}
	Recall for the following proposition (see the discussion preceding Theorem \ref{thm: Koszul of gp like alg in Lurie}) that for an $\infty$-category $\Cinf$ with all finite products, we use $\Grp_{\Ek}(\Cinf)$ to denote the category of \textit{grouplike} $\Ek$-monoids in $\Cinf$, as defined in \cite[5.2.6.6]{ha}.
\end{rem}

\begin{prop}\label{prop: truncated groups}
For a presentable $\infty$-category $\Cinf$, there is an equivalence of $\infty$-categories: $\tau_{\leq n}\Grp_{\Ek}(\Cinf)\simeq \Grp_{\Ek}(\tau_{\leq n}\Cinf)$. 
\end{prop}

\begin{proof}
This proposition could be proven similarly to Lemma \ref{lem: lmodequiv} but it admits another somewhat more interesting proof. Note that $\Ek$-group objects in $\Tinf$ are monadic over $\Tinf$. This follows from an application of the Barr-Beck theorem which, in the case of $\infty$-categories, is \cite[Theorem 4.7.3.5]{ha}. From \cite[Corollary 5.2.6.18]{ha} we have that the forgetful functor $\mathcal{Grp}_{\Ek}(\Tinf)\to\Tinf$ is conservative and preserves sifted colimits, so in particular preserves geometric realizations (cf.~\cite[Corollary 5.5.8.4]{htt}). We additionally need the somewhat classical fact that $k$-connective spaces are closed under geometric realization (see, for instance, \cite[Lemma 2.4]{ebertsemisimplicial}).

Now notice that, again since the forgetful functor preserves sifted colimits, $\Ek$-group objects in $\Tinf$ are in fact \textit{algebraic} over $\Tinf$ by \cite[Theorem B.7]{gepnerinfiniteloopspacemachines}. In other words, $\Ek$-group objects in $\Tinf$ are the algebras for an $\infty$-categorical Lawvere theory (an explicit description of this theory when $k=1$ is given in \cite[Corollary 5.2.6.21]{ha}). From \cite[Proposition B.3]{gepnerinfiniteloopspacemachines} it then follows that $\mathcal{Grp}_{\Ek}(\Cinf)\simeq\Cinf\otimes\mathcal{Grp}_{\Ek}(\Tinf)$, where the tensor product is that of presentable $\infty$-categories, $\mathcal{Pr}^L$, described in \cite[Proposition  4.8.1.15]{ha}. Combining this with the facts that $\tau_{\leq n}\Dinf\simeq\Dinf\otimes\tau_{\leq n}\Tinf$ for any presentable $\infty$-category $\Dinf$ (cf.~\cite[Example 4.8.1.22]{ha}) and that the tensor product of $\mathcal{Pr}^L$ is associative and commutative, we have:
\begin{align*}
\tau_{\leq n}\mathcal{Grp}_{\Ek}(\Cinf) & \simeq \mathcal{Grp}_{\Ek}(\Cinf)\otimes\tau_{\leq n}\Tinf\\
&\simeq \Cinf\otimes\mathcal{Grp}_{\Ek}(\Tinf)\otimes\tau_{\leq n}\Tinf\\
&\simeq \Cinf\otimes\tau_{\leq n}\Tinf\otimes\mathcal{Grp}_{\Ek}(\Tinf)\\
&\simeq \tau_{\leq n}\Cinf\otimes\mathcal{Grp}_{\Ek}(\Tinf)\\
&\simeq \mathcal{Grp}_{\Ek}(\tau_{\leq n}\Cinf). \qedhere
\end{align*}
\end{proof}

\begin{rem}
It should be possible to prove Lemma \ref{lem: lmodequiv} from the Lawvere theory viewpoint of the proof of Proposition \ref{prop: truncated groups}. Unfortunately there does not yet exist, to the knowledge of the authors, a theory of multi-sorted $\infty$-categorical Lawvere theories, which would be necessary to discuss the theory whose algebras are pairs $(A,M)$ where $A$ is a monoid and $M$ is a module over it. 
\end{rem}

We now state the result which is an $n$-toposic version of Lurie's Theorem \ref{thm: Koszul of gp like alg in Lurie}.

\begin{thm}\label{cor: cobar in n-topoi}
Let $\XX$ be an $n$-topos endowed with the Cartesian symmetric mono\-idal structure. 
Then the cobar construction induces an equivalence between $k$-connective coaugmented $\mathbb{E}_\infty$-coalgebras in $\XX$ and $\Ek$-group objects in $\tau_{\leq n-k-1}\XX$:
\[
\C^k:\XX_\ast^{\geq k}\simeq\coalginf^{coaug}_{\mathbb{E}_\infty}(\XX^{\geq k})\stackrel{\simeq}\longrightarrow \mathcal{Grp}_{\Ek}(\tau_{\leq n-k-1}\XX).
\]
\end{thm}

\begin{proof}
Let $\Xh$ be an $\infty$-topos whose truncation $\tau_{\leq n-1}\Xh$ is equivalent to the $n$-topos $\XX$. By Lemma \ref{lem: truncatedequivalence}, the following diagram commutes up to equivalence:

\[
\begin{tikzcd}
\mathscr{Y}_\ast^{\geq k}  \arrow[r, "\C^k"] \arrow[d, swap, "\tau_{\leq n-k-1}^{\mathscr{Y}^{\geq k}_\ast}" ] &  \mathcal{Grp}_{\Ek}(\mathscr{Y})\arrow[d,"\tau_{\leq n-k-1}^{\mathcal{Grp}_{\Ek}(\mathscr{Y})}"] \\
\tau_{\leq n-k-1}(\mathscr{Y}_\ast^{\geq k}) \arrow[r, "\C^k"] &  \tau_{\leq n-k-1}\mathcal{Grp}_{\Ek}(\mathscr{Y}).
\end{tikzcd}
\]

Moreover the bottom horizontal functor $\C^k$ is still an equivalence with inverse equivalence $\B^k$ (suitably restricted). Additionally, by Proposition \ref{prop: truncated groups}, the bottom right hand corner of the above diagram is equivalent to $\mathcal{Grp}_{\Ek}(\tau_{\leq n-k-1}\XX)$. Thus it only remains to show that the bottom left hand corner of the above diagram, $\tau_{\leq n-k-1}(\mathscr{Y}^{\geq k}_\ast)$, is equivalent to $(\tau_{\leq n-1}\YY)^{\geq k}_\ast\simeq\XX^{\geq k}_\ast$. Note that taking the full subcategory of $(n-k-1)$-truncated objects of $\YY^{\geq k}_\ast$ is not equivalent to taking pointed connected objects in the subcategory of $(n-k-1)$-truncated objects. 

The $\infty$-category of pointed $k$-connective objects of $\XX$, $\XX_\ast^{\geq k}$, can be defined as the pullback of the forgetful functor $\YY^{\geq k}_\ast\to\YY$ along the inclusion of the full subcategory $\XX\simeq \tau_{\leq n-1}\YY\hookrightarrow \mathscr{Y}$:

\[\begin{tikzcd}
	{\mathscr{X}^{\geq k}_\ast} & {\mathscr{Y}^{\geq k}_\ast} \\
	{\mathscr{X}} & {\mathscr{Y}}
	\arrow["{i}", from=2-1, to=2-2, hook]
	\arrow["{U}", from=1-2, to=2-2]
	\arrow[from=1-1, to=1-2, hook]
	\arrow["{U}"', from=1-1, to=2-1]
	\arrow["\lrcorner"{pos=0, rotate=0}, from=1-1, to=2-2, phantom].
\end{tikzcd}\]
In other words, the full subcategory of $\YY^{\geq k}_\ast$ on those objects whose underlying $\YY$-object is $(n-1)$-truncated is equivalent to the $\infty$-category of pointed $k$-connective objects of $\XX$. Recall that fully faithful functors form the right class of an orthogonal factorization system on $\Catinf$, so they are stable under pullback (cf.~\cite[24.12]{joyalnotes}). 
Therefore the induced map $\XX^{\geq k}_\ast\hookrightarrow \YY^{\geq k}_\ast$ is full and faithful.  

By definition there is an inclusion of a full subcategory $\tau_{\leq n-k-1}(\YY^{\geq k}_\ast)\hookrightarrow\YY^{\geq k}_\ast$. If we compose this functor with the forgetful functor $\YY^{\geq k}_\ast\to \YY$, it factors through $\XX\simeq \tau_{\leq n-1}\YY$. To see this, note that given any $X\in\tau_{\leq n-k-1}(\YY^{\geq k}_\ast)$ we have $X\simeq \B^k\C^k X$. Since $X\in\tau_{\leq n-k-1}(\YY^{\geq k}_\ast)$, we know that $\C^k X\in\tau_{\leq n-k-1}\mathcal{Grp}_{\Ek}(\YY)\simeq \mathcal{Grp}_{\Ek}(\tau_{\leq n-k-1}\YY)$, i.e.~$\Omega^k X$ is $(n-k-1)$-truncated in $\YY$. Then by Lemma \ref{lem: loopsBshifting} we know that $\B^k\C^k X$ is $(n-1)$-truncated. This factorization and the above described inclusion then imply, by the universal property of the pullback, the existence of the dotted functor below:

\[\begin{tikzcd}
	{\tau_{\leq n-k-1}(\mathscr{Y}^{\geq k}_\ast)} \\
	& {\mathscr{X}^{\geq k}_\ast} & {\mathscr{Y}^{\geq k}_\ast} \\
	& {\mathscr{X}} & {\mathscr{Y}}
	\arrow["{i}", from=3-2, to=3-3, hook]
	\arrow["{U}", from=2-3, to=3-3]
	\arrow[from=2-2, to=2-3, hook]
	\arrow["{U}"', from=2-2, to=3-2]
	\arrow["\lrcorner"{very near start, rotate=0}, from=2-2, to=3-3, phantom]
	\arrow[from=1-1, to=2-2, dotted]
	\arrow[from=1-1, to=2-3, curve={height=-12pt}, hook]
	\arrow[from=1-1, to=3-2, curve={height=12pt}]
\end{tikzcd}\]
Again using the fact that fully faithful functors of $\infty$-categories are the right class of an orthogonal factorization system and therefore satisfy a 2-out-of-3 property, the dotted arrow above is full and faithful. In other words, $\tau_{\leq n-k-1}(\YY^{\geq k}_{\ast})$ is (equivalent to) a full subcategory of $\XX_\ast^{\geq k}$. It then remains to show that this functor is also essentially surjective. 

We need to show that, given an object $X\in\XX_\ast^{\geq k}$, it is equivalent to an object of $\tau_{\leq n-k-1}(\YY_\ast^{\geq k})$. To do this, we first consider $X$ as an object of $\YY_\ast^{\geq k}$ i.e.~as a pointed connected object of $\YY$ whose underlying $\YY$-object is $(n-1)$-truncated. Therefore $\C^k X\in \mathcal{Grp}(\tau_{n-k-1}\XX)\subseteq\mathcal{Grp}(\YY)$ (by Lemma \ref{lem: loopsBshifting}). It then follows that $X\simeq \B^k\C^k X\in\tau_{\leq n-k-1}(\YY^{\geq k}_\ast)\subseteq\YY^{\geq k}_\ast$. 
\end{proof}

\begin{cor}\label{cor: intervaltypes}
	Let $a<b$ be non-negative integers, and let $[a,b]$ denote the interval of integers between them, including the endpoints. Then there is an equivalence of $\infty$-categories of $a$-connective $b$-groupoids (i.e.~spaces whose homotopy groups are concentrated in degrees $a$ through $b$), and $\mathbb{E}_a$-group objects in $\tau_{\leq b-a}\Tinf$. 
\end{cor}

\begin{proof}
	This is the application of Theorem \ref{cor: cobar in n-topoi} in the case that $n=b+1$ and $k=a$ in Theorem \ref{cor: cobar in n-topoi}.
\end{proof}

If we refer to the objects of $\tau_{\leq b}\Tinf^{\geq a}$ as \textit{homotopy $[a,b]$-types} then Theorem \ref{cor: cobar in n-topoi} and Corollary \ref{cor: intervaltypes} can be interpreted as giving ``algebraic models'' for homotopy $[a,b]$-types (or sheaves thereof), i.e.~homotopy types whose homotopy groups are concentrated in some interval $[a,b]$ of natural numbers. A list of classical results of this type can be found, for instance, in the introduction of \cite{garzonMirandaSerre}. However, this statement needs to be taken with a grain of salt. For instance, one such result is the identification of the homotopy category of $[0,2]$-types with either a derived category of crossed modules or a derived category of (strict, discrete) 2-groupoids (cf.~\cite{whiteheadCombHomI,whiteheadCombHomII,noohi2groupoids}). The content of Theorem \ref{cor: cobar in n-topoi} in this case is trivial. Up to equivalence of $\infty$-categories, it identifies the 3-topos $\tau_{\leq 2}\Tinf$ with itself. In other words, the description of $[0,2]$-types as, say, certain kinds of crossed modules, cannot be recovered automatically from the above theorem. This is the case for $[0,b]$-types for any $b\geq 0$. When $a\neq 0$ however, slightly more interesting descriptions result, some of which we document here.

\begin{ex}
Setting $n=2$, $k=1$ in Theorem \ref{cor: cobar in n-topoi} gives the well known looping-delooping equivalence between $\mathbb{E}_1$-objects in $\tau_{\leq 0}\Tinf$, i.e.~discrete groups, and pointed $[1,1]$-types.
\end{ex}

\begin{ex}\label{ex: browngilbert}
Setting $n=4$, $k=2$ gives an equivalence between pointed $[2,3]$-types and $\mathbb{E}_2$-groups in groupoids. The latter are sometimes referred to as \textit{braided categorical groups} in the literature, and this is a classical result (the authors have not been able to find a reference containing a proof of this fact, but it is stated in the introduction of \cite{garzonMirandaCatGroups} and attributed to Brown and Gilbert \cite{brownGilbertAlgebraic}). Similarly to the case of $[0,n]$-types however, our theorem does not recover the description of such objects in terms of explicit braided crossed modules of groups as in \cite{brownGilbertAlgebraic}. 
\end{ex}

\begin{ex}\label{ex: bullejos}
	Consider pointed $[5,6]$-types, so that $k=5$ and $n=7$. Then we obtain an equivalence between such spaces and $\mathbb{E}_5$-groups in groupoids. Of course groupoids form a 2-topos, hence are a 2-category and by the Baez-Dolan Stabilization Hypothesis (cf.~\cite[Example 5.1.2.3]{ha} for a description of this in the context of $\infty$-categories), $\mathbb{E}_5$-groups in groupoids are symmetric monoidal groupoids. In general, whenever $n>2$, the $\infty$-category of pointed $[n,n+1]$-types is equivalent to the $\infty$-category of $\mathbb{E}_\infty$-groups in groupoids. This should be compared with the main theorem of \cite{bullejosSymmetricCatGroups}, which says that $[n,n+1]$-types with $n>2$ are modeled by symmetric monoidal groupoids whose objects form a group.
\end{ex}

\begin{ex}
	Similarly to the above example, if $a>b-a+1$, then $[a,b]$-types are equivalent to infinite loop spaces whose homotopy groups are concentrated in degrees 0 through $b-a$.
\end{ex}

\begin{rem}
	One addition to the literature that our theorem does provide is a general description of \textit{(pre)sheaves} of $[a,b]$-types. All of the classical results described above only consider subcategories of $\Tinf$. 
\end{rem}

\section{Koszul Duality in Higher Topoi for Comodules and Modules}

We now state a ``categorification'' of Theorem \ref{cor: cobar in n-topoi} in that it relates $\infty$-categories of modules and comodules rather than algebras and coalgebras. 

\begin{thm}\label{theorem: mainthmn}\label{theorem: mainthminfty}
Let $\infty\geq n \geq 0$ and let  $\XX$ be an $n$-topos equipped with the Cartesian symmetric monoidal structure.
If $X$ is a pointed and $k$-connective object of $\XX$, then there is an equivalence of $\infty$-categories: 
	\[\mathcal{LCoMod}_{\Omega^{k-1}X}(\XX) \simeq\mathcal{LMod}_{\Omega^{k} X}(\XX).\] 
	Furthermore, if $X\simeq\pi^\ast(Y)$ for a pointed $(n-1)$-groupoid $Y$ such that the pointing of $X$ is equivalent to $1_\XX\simeq\pi^\ast(1_{\tau_{\leq n-1}\Tinf})\to \pi^\ast(Y)\simeq X$, both of the above $\infty$-categories are equivalent to $\Fun(\Omega^{k-1}Y,\XX)$. 
\end{thm}

\begin{rem}
	Note that, in the above theorem, we could have equivalently written $\mathcal{RCoMod}$ or $\mathcal{CoMod}$ in place of $\mathcal{LCoMod}$ because every object in the Cartesian monoidal structure is a cocommutative coalgebra (see Corollary \ref{corollary: slicecomodules} below), but we cannot make the same simplifications for the $\infty$-categories of modules. 
\end{rem}

\noindent We provide some examples below of the implications of this theorem in familiar settings:

\begin{ex}
	Let $\X$ be a $1$-topos of sheaves of sets on a site. Then the only connected object is the constant sheaf valued in the one element set, i.e.~the terminal object $1_\X$. This sheaf is also the pullback of the one element set along the geometric morphism $\pi\colon \X\to \mathcal{Set}$, so we have the (unsurprising) equivalence: \[
	\X\simeq\X_{/1_\X}\simeq\comodinf_{1_\X}(\X)\simeq \lmodinf_{1_\X}(\X).
	\]
	In other words, Koszul duality is not interesting in the classical, underived setting. 
\end{ex}

\begin{ex}
	Let $\mathcal{Gpd}$ be the $2$-topos of groupoids (i.e.~homotopy 1-types). Then the $1$-connective objects are precisely the objects of the form $BG$, or $K(G,1)$, for $G$ a discrete group, thought of as a one object groupoid with morphisms equivalent to $G$. Then Theorem \ref{theorem: mainthminfty} identifies the $\infty$-category of groupoids over $BG$ with the $\infty$-category of groupoids with $G$-action. A similar statement holds for the $2$-topos of sheaves of groupoids (i.e.~stacks) on any site. 
\end{ex}

\begin{ex}
	Let $\X$ be an $\infty$-topos of simplicial sheaves on a category equipped with a Grothendieck topology and let $X$ be any space. Suppose $\pi\colon\X\to\Tinf$ is the unique geometric morphism. Then there is an equivalence of $\infty$-categories between sheaves in $\X$ with a morphism to the constant sheaf $\pi^\ast(X)$ and sheaves in $\X$ with an action by the loop space $\Omega X$. 
\end{ex}

{We prove Theorem \ref{theorem: mainthminfty}  using several of the lemmas below. The next result is a restatement of Lurie's.}

\begin{lem}[{\cite[5.2.6.28, 5.2.6.29]{ha}}]\label{lemma: dkfresult}
Let $X$ be a pointed and connected object of an $\infty$-topos $\X$. Then there is an equivalence of $\infty$-categories: \[\bar{\theta}\colon \X_{/X}\to \lmodinf_{\Omega X}(\X).\] Moreover, pulling back along the pointing $\ast\to X$ is equivalent to the composite of the forgetful functor $\lmodinf_{\Omega X}(\X)\to \X$ with $\bar{\theta}$
\end{lem}

The proof of the following application of Lemma \ref{lemma: dkfresult} was suggested to us by Peter Haine.

\begin{cor}\label{cor: dkfresult of n topoi}
	Let $X$ be a pointed and connected object of an $n$-topos $\XX$. Then there is an equivalence of $\infty$-categories: \[\lmodinf_{\Omega X}(\XX)\simeq \XX_{/X}.\]
\end{cor}

\begin{proof}
Let $\Xh$ be the $n$-localic $\infty$-topos whose $n$-truncation is equivalent to $\XX$, 
i.e.~$\tau_{\leq n-1}\Xh\simeq \XX$. By Lemma \ref{lemma: dkfresult} we have an equivalence: \[\lmodinf_{\Omega X}(\Xh)\simeq \Xh_{/X}.\] By Lemma \ref{lem: lmodequiv} there is an equivalence $\tau_{\leq n-1}\lmodinf_{\Omega X}(\YY)\simeq \lmodinf_{\Omega X}(\XX)$. Therefore it only remains to show that $\tau_{n-1}(\YY_{/X})\simeq \XX_{/X}$. 

Recall that morphism spaces in a slice $\infty$-category are computed as pullbacks. In our case, suppose $f:Z\to X$ and $g\colon Y\to X$ are objects of $\YY_{/X}$. Then the space of morphisms from $f$ to $g$ is equivalent to, as shown in \cite[Lemma 5.5.5.12]{htt}, the upper left corner of the following pullback diagram:

\[\begin{tikzcd}
	{\mathscr{Y}_{/X}(f,g)} & {\mathscr{Y}(Z,Y)} \\
	{\ast} & {\mathscr{Y}(Z,X)}
	\arrow["{g\circ -}", from=1-2, to=2-2]
	\arrow["{\{f\}}"', from=2-1, to=2-2]
	\arrow[from=1-1, to=2-1]
	\arrow[from=1-1, to=1-2]
	\arrow["\lrcorner"{pos=0, rotate=0}, from=1-1, to=2-2, phantom]
\end{tikzcd}\]

By considering the long exact sequence in homotopy groups associated to the above fiber sequence of spaces one readily sees that if $X$ is $(n-1)$-truncated then $Y$ is $(n-1)$-truncated in $\YY$ if and only if $g\colon Y\to X$ is $(n-1)$-truncated in $\YY_{/X}$ (c.f.~\cite[Lemma 5.5.6.14]{htt}). So the full subcategory of $(n-1)$-truncated objects of $\YY_{/X}$ is precisely the full subcategory spanned by morphisms $Y\to X$ whose domain is $(n-1)$-truncated in $\YY$ or, in other words, $\XX_{/X}$. 
\end{proof}

\begin{rem}
From Lemma \ref{lemma: dkfresult} it follows that one direction of the equivalence given in Corollary \ref{cor: dkfresult of n topoi}, $\XX_{/X}\to \lmodinf_{\Omega X}(\Xh)$, is again given on underlying objects by taking the fiber along the pointing $\ast\to X$. 
\end{rem}

We now detour onto the special case that our object $X$ is the pullback of an $\infty$-groupoid, i.e.~$X\simeq \pi^\ast(Y)$ for $Y\in\Tinf$.

\begin{lem}\label{lemma:modules over loops are functors}
For a connected $\infty$-groupoid $X$ and a presentable $\infty$-category $\Cinf$ there is an equivalence of $\infty$-categories $\lmodinf_{\Omega X}(\Cinf)\simeq \Fun(X,\Cinf)$. 
\end{lem}

\begin{proof}
		As a result of  \cite[Theorem 4.8.4.1]{ha} we have an equivalence: $$Lin\Fun^{L}_{\Tinf}(\rmodinf_{\Omega X}(\Tinf),\Cinf)\simeq \lmodinf_{\Omega X}(\Cinf),$$ where $Lin\Fun^{L}_{\Tinf}(\rmodinf_{\Omega X}(\Tinf),\Cinf)$ is the $\infty$-category of colimit preserving functors from $\rmodinf_{\Omega X}(\Tinf)$ to $\Cinf$ that are linear with respect to the tensoring over $\Tinf$. 
		Note however that the functors $\rmodinf_{\Omega X}(\Tinf)\to \Cinf$ which are $\Tinf$-linear, in the sense of \cite[Definition 4.5.2.7]{ha}, are exactly the functors preserving colimits of functors indexed by $\infty$-groupoids (this follows from the fact that both $\Tinf$ and $\Cinf$ are tensored over $\Tinf$ by taking colimits, as in  \cite[Remark 5.5.1.7]{htt}), so $Lin\Fun^{L}_{\Tinf}(\rmodinf_{\Omega X}(\Tinf),\Cinf)$ is exactly the $\infty$-category of colimit preserving functors from $\rmodinf_{\Omega X}(\Tinf)$ to $\Cinf$, i.e.~there is an equivalence of $\infty$-categories: $$Lin\Fun^{L}_{\Tinf}(\rmodinf_{\Omega X}(\Tinf),\Cinf)\simeq \Fun^L(\rmodinf_{\Omega X}(\Tinf),\Cinf).$$
		
		By Lemma \ref{lemma: dkfresult} there is an equivalence: $\rmodinf_{\Omega X}(\Tinf)\simeq \Tinf_{/X}.$ By the straightening construction of \cite{htt}, there is an equivalence   $\Tinf_{/X}\simeq \Fun(X^{\op},\Tinf)$. Because $X$ is an $\infty$-groupoid, there is a canonical equivalence $X^{\op}\simeq X$, which gives the equivalence $\Tinf_{/X}\simeq \Fun(X,\Tinf)$. 
		By \cite[Theorem 5.1.5.6]{htt} we have that colimit preserving functors out of $\Fun(X,\Tinf)$ are equivalent to functors out of $X$, i.e.~(since $\Cinf$ is presentable and therefore cocomplete) there is an equivalence: $$\Fun^L(\Fun(X,\Tinf),\Cinf)\simeq \Fun(X,\Cinf).$$
		Therefore, we have an equivalence: $$\mathcal{LMod}_{\Omega X}(\Cinf)\simeq Lin\Fun^{L}_{\Tinf}(\mathcal{RMod}_{\Omega X}(\Tinf),\Cinf)\simeq \Fun(X,\Cinf),$$ which completes the proof. 
\end{proof}

\begin{lem}\label{lemma: gengrothconstruction}
	For an $\infty$-topos $\X$ and a connected $\infty$-groupoid $X$, there is an equivalence of $\infty$-categories $\Fun(X,\X)\simeq \X_{/\pi^\ast(X)}$, where $\pi^\ast\colon \Tinf\to\XX$ is the left adjoint in the canonical geometric morphism from $\XX$ to $\Tinf$. 
\end{lem}

\begin{proof}
We show that both $\Fun(X,\X)$ and $\X_{/\pi^\ast(X)}$ are equivalent to $lim_X\X$, the limit in the $\infty$-category of (small) $\infty$-categories, $\Catinf$, of the constant diagram valued in $\X$. 
Since $X\simeq colim_X 1_{\Catinf}$ is the colimit of the constant diagram valued in the terminal Kan complex (thought of as an $\infty$-category), we have, in $\mathcal{Cat}_\infty$, an equivalence: 
\[\Fun(X,\X)\simeq \Fun(colim_X 1_{\Catinf},\X)\simeq lim_X \Fun(1_{\Catinf},\X)\simeq lim_X\X.\]  

Now recall from \cite[Lemma 6.1.1.1]{htt} that there is a Cartesian (and coCartesian) fibration $p\colon \mathcal{O}_{\X}\to \X$ whose fibers over $U\in\X$ are the slice $\infty$-topoi $\X_{/U}$. 
This fibration has an associated functor (by straightening) $F_\X\colon \X^{\op}\to \Catinf$, $U\mapsto \X_{/U}$.
By applying \cite[Proposition 6.1.3.10]{htt} to \cite[Theorem 6.1.3.9 (3)]{htt}, we get that $F_\X$ takes colimits in $\X$ to limits in $\Catinf$ (in fact in $Pr^L$, the sub-$\infty$-category of \textit{presentable} $\infty$-categories and left adjoints). Recall that the pullback along the geometric morphism $\pi\colon \X\to \Tinf$ preserves colimits and terminal objects (cf.~\cite[Definition 6.3.1.5]{htt}), so $\pi^\ast(X)\simeq\pi^\ast(colim_X1_\Tinf)\simeq colim_X\pi^\ast(1_\Tinf)\simeq colim_X1_{\X}$. So we have that $F_{\X}(\pi^\ast(X))\simeq \X_{/\pi^\ast(X)}\simeq lim_X\X_{/1_{\X}}\simeq lim_X\X$. This completes the proof.
\end{proof}

\begin{lem}\label{lemma: ngengrothconstruction}
	For an $n$-topos $\XX$ and a connected $(n-1)$-groupoid $X$, there is an equivalence of $\infty$-categories $\Fun(X,\XX)\simeq \XX_{/\pi^\ast(X)}$ where $\pi^\ast$ is left adjoint in the canonical geometric morphism from $\XX$ to $\tau_{\leq n-1}\Tinf$. 	
\end{lem}

\begin{proof}
	The proof proceeds identically to that of Lemma \ref{lemma: gengrothconstruction} except that we must apply \cite[Propositions 6.4.4.6, 6.4.4.7]{htt} to \cite[Theorem 6.4.4.5 (3)]{htt}, instead of \cite[Proposition 6.1.3.10]{htt} to  \cite[Theorem 6.1.3.9]{htt}. It is helpful to notice that, in the application of \cite[Theorem 6.4.4.5]{htt} here, every diagram indexed by a connected $(n-1)$-groupoid (which is just a certain kind of $\infty$-groupoid) consists of $(-2)$-truncated morphisms, i.e.~equivalences, because every morphism in an $\infty$-groupoid is an equivalence and there is an equivalence between every vertex of a connected $\infty$-groupoid. 
		\end{proof}

\begin{rem}
	The above lemmas may be thought of as a generalized straightening and unstraightening, or the Grothendieck construction and its inverse, for $n$-topoi for $0\leq n\leq\infty$. In \cite[Theorem 3.4.20]{schreiberdifferentialcoh}, Schreiber describes a special case of our Lemma \ref{lemma: gengrothconstruction} which only holds for $\infty$-topoi with the additional (and somewhat strong) property of having an \emph{$\infty$-connected site of definition}. The reader is referred to the above citation for definitions and further results on such $\infty$-topoi. 
\end{rem}

\begin{proof}[Proof of Theorem \ref{theorem: mainthminfty}]
We first prove the result for $k=1$. Let $X$ be a pointed connected object in $\X$. The first desired statement is a combination of Corollary \ref{cor: dkfresult of n topoi} (or Lemma \ref{lemma: dkfresult} when $n=\infty$) and Corollary \ref{corollary: slicecomodules} where $\Cinf=\X$:
\[
\mathcal{LCoMod}_{X}(\X) \simeq \X_{/X}\simeq \mathcal{LMod}_{\Omega X}(\X).
\]
The second statement for $X\simeq \pi^*(Y)$, where $Y$ is a pointed $(n-1)$-groupoid, follows from Lemma \ref{lemma: ngengrothconstruction} (or Lemma \ref{lemma: gengrothconstruction} when $n=\infty$).

Now assume $1 < k < \infty$. Let $X$ be a pointed $k$-connective object in  $\X$. Define $X':= \Omega^{k-1}X$. Since $X'$ is $1$-connective (i.e.~connected) if and only if $X$ is $k$-connective, we obtain directly:
\[
\mathcal{LCoMod}_{X'}(\X)\simeq \X_{/X'}\simeq \mathcal{LMod}_{\Omega X'}(\X),
\]
from the case $k=1$. 
Moreover, if $X\simeq \pi^*(Y)$ for some $(n-1)$-groupoid $Y$, then we get:
\[
X'= \Omega^{k-1}X \simeq \Omega^{k-1} \pi^* (Y) \simeq \pi^*(\Omega^{k-1} Y),
\]
as the functor $\pi^*: \tau_{\leq n-1}\mathcal{S}\rightarrow \X$ preserves finite limits. Hence we apply again Lemma \ref{lemma: ngengrothconstruction} (or Lemma \ref{lemma: gengrothconstruction} if $n=\infty$) to conclude.
\end{proof}

\renewcommand{\bibname}{References}
\bibliographystyle{amsalpha}
\bibliography{biblio}

\providecommand{\bysame}{\leavevmode\hbox to3em{\hrulefill}\thinspace}
\providecommand{\MR}{\relax\ifhmode\unskip\space\fi MR }
\providecommand{\MRhref}[2]{%
  \href{http://www.ams.org/mathscinet-getitem?mr=#1}{#2}
}
\providecommand{\href}[2]{#2}
\begin{thebibliography}{{Sch}13}

\bibitem[BCC93]{bullejosSymmetricCatGroups}
M.~Bullejos, P.~Carrasco, and A.~M. Cegarra, \emph{Cohomology with coefficients
  in symmetric cat-groups. {A}n extension of {E}ilenberg-{M}ac {L}ane's
  classification theorem}, Math. Proc. Cambridge Philos. Soc. \textbf{114}
  (1993), no.~1, 163--189. \MR{1219923}

\bibitem[BG89]{brownGilbertAlgebraic}
Ronald Brown and N.~D. Gilbert, \emph{Algebraic models of {$3$}-types and
  automorphism structures for crossed modules}, Proc. London Math. Soc. (3)
  \textbf{59} (1989), no.~1, 51--73. \MR{997251}

\bibitem[DDK80]{drordwyerkan}
E.~Dror, W.~G. Dwyer, and D.~M. Kan, \emph{Equivariant maps which are self
  homotopy equivalences}, Proc. Amer. Math. Soc. \textbf{80} (1980), no.~4,
  670--672. \MR{587952}

\bibitem[ERW19]{ebertsemisimplicial}
Johannes Ebert and Oscar Randal-Williams, \emph{Semisimplicial spaces}, Algebr.
  Geom. Topol. \textbf{19} (2019), no.~4, 2099--2150.

\bibitem[GGN15]{gepnerinfiniteloopspacemachines}
David Gepner, Moritz Groth, and Thomas Nikolaus, \emph{Universality of
  multiplicative infinite loop space machines}, Algebr. Geom. Topol.
  \textbf{15} (2015), no.~6, 3107--3153.

\bibitem[GJ09]{goerssjardine}
Paul~G. Goerss and John~F. Jardine, \emph{Simplicial homotopy theory}, Modern
  Birkh\"{a}user Classics, Birkh\"{a}user Verlag, Basel, 2009, Reprint of the
  1999 edition [MR1711612]. \MR{2840650}

\bibitem[GM97]{garzonMirandaCatGroups}
Antonio~R. Garzon and Jesus~G. Miranda, \emph{Homotopy theory for (braided)
  {${\rm CAT}$}-groups}, Cahiers Topologie G\'{e}om. Diff\'{e}rentielle
  Cat\'{e}g. \textbf{38} (1997), no.~2, 99--139. \MR{1454159}

\bibitem[GM00]{garzonMirandaSerre}
A.~R. Garz\'{o}n and J.~G. Miranda, \emph{Serre homotopy theory in
  subcategories of simplicial groups}, J. Pure Appl. Algebra \textbf{147}
  (2000), no.~2, 107--123. \MR{1747619}

\bibitem[HMS19]{haugsengshifted}
Rune {Haugseng}, Valerio {Melani}, and Pavel {Safronov}, \emph{{Shifted
  Coisotropic Correspondences}}, arXiv e-prints (2019), arXiv:1904.11312.

\bibitem[HS16]{hscomod}
Kathryn Hess and Brooke Shipley, \emph{Waldhausen {$K$}-theory of spaces via
  comodules}, Adv. Math. \textbf{290} (2016), 1079--1137. \MR{3451948}

\bibitem[Joy08]{joyalnotes}
Andr\'e Joyal, \emph{Notes on quasi-categories}, 2008,
  https://web.math.rochester.edu/people/
  faculty/doug/otherpapers/Joyal-QC-Notes.pdf.

\bibitem[Lur09]{htt}
Jacob Lurie, \emph{Higher topos theory}, Annals of Mathematics Studies, vol.
  170, Princeton University Press, Princeton, NJ, 2009. \MR{2522659}

\bibitem[Lur17]{ha}
\bysame, \emph{Higher algebra},
  \url{http://www.math.harvard.edu/~lurie/papers/HA.pdf}, 2017, electronic
  book.

\bibitem[May72]{maygeom}
J.~P. May, \emph{The geometry of iterated loop spaces}, Springer-Verlag,
  Berlin-New York, 1972, Lectures Notes in Mathematics, Vol. 271. \MR{0420610}

\bibitem[Noo07]{noohi2groupoids}
Behrang Noohi, \emph{Notes on 2-groupoids, 2-groups and crossed modules},
  Homology Homotopy Appl. \textbf{9} (2007), no.~1, 75--106. \MR{2280287}

\bibitem[{Sch}13]{schreiberdifferentialcoh}
Urs {Schreiber}, \emph{{Differential cohomology in a cohesive $\infty$-topos}},
  arXiv e-prints (2013), arXiv:1310.7930.

\bibitem[Shu08]{shulmanparametrized}
Michael~A. Shulman, \emph{Parametrized spaces model locally constant homotopy
  sheaves}, Topology Appl. \textbf{155} (2008), no.~5, 412--432. \MR{2380927}

\bibitem[Whi49a]{whiteheadCombHomI}
J.~H.~C. Whitehead, \emph{Combinatorial homotopy. {I}}, Bull. Amer. Math. Soc.
  \textbf{55} (1949), 213--245. \MR{30759}

\bibitem[Whi49b]{whiteheadCombHomII}
\bysame, \emph{Combinatorial homotopy. {II}}, Bull. Amer. Math. Soc.
  \textbf{55} (1949), 453--496. \MR{30760}

\end{thebibliography}
\end{document}